\documentclass[12pt,a4paper,reqno]{amsart}
\usepackage[utf8]{inputenc}
\usepackage[T1]{fontenc}
\usepackage{mathptmx}
\usepackage{amsfonts}

\usepackage[english]{babel}
\usepackage{amsmath, amsthm, amsfonts,amssymb}

\newtheorem{thm}{Theorem}[section]
\newtheorem{cor}[thm]{Corollary}
\newtheorem{lem}[thm]{Lemma}
\newtheorem{prop}[thm]{Proposition}
\newtheorem{rem}[thm]{Remark}

\theoremstyle{plain}

\theoremstyle{definition}

\newtheorem{defn}[thm]{Definition}
\theoremstyle{remark}

\newcommand\blfootnote[1]{%
	\begingroup
	\renewcommand\thefootnote{}\footnote{#1}%
	\addtocounter{footnote}{-1}%
	\endgroup
}

\newcommand{\N}{\mathbb{N}}

\newcommand{\X}{\mathbb{X}}

\def\ba{\begin{eqnarray*}}
\def\ea{\end{eqnarray*}}

\def\bee{\begin{equation}}

\def\ene{\end{equation}}

\newcommand{\sgn}{\mathrm{sgn \,}}
\newcommand{\supp}{\mathrm{supp \,}}
\title{Equivalence between almost-greedy and semi-greedy bases}
\date{}
\begin{document}
			\author{P. M. Bern\'a}
		\address{Pablo M. Bern\'a
			\\
			Departmento de Matem\'aticas
			\\
			Universidad Aut\'onoma de Madrid
			\\
			28049 Madrid, Spain} \email{pablo.berna@uam.es}
		\maketitle

		\begin{abstract} In \cite{DKK} it was proved that almost-greedy and semi-greedy bases are equivalent in the context of Banach spaces with finite cotype. In this paper we show this equivalence for general Banach spaces.\end{abstract}
\blfootnote{\hspace{-0.031\textwidth} 2000 Mathematics Subject Classification. 46B15, 41A65.\newline
		The author was supported by a PhD fellowship FPI-UAM and the grants MTM-2016-76566-P (MINECO, Spain) and 19368/PI/14 (\emph{Fundaci\'on S\'eneca}, Regi\'on de Murcia, Spain).  }
		
\begin{section}{INTRODUCTION}
	Let $(\X,\Vert \cdot \Vert)$ be a Banach space over $\mathbb F$ ($\mathbb F$ denotes the real field $\mathbb R$ or the complex field $\mathbb C$) and let $\mathcal{B}=(e_n)_{n=1}^{\infty}$ be a semi-normalized Schauder basis of $\X$ with constant $K_b$ and with biorthogonal functionals $(e_n^{*})_{n=1}^{\infty}$, i.e, $0<\inf_n\Vert e_n\Vert \leq \sup_n\Vert e_n\Vert<\infty$ and $K_b=\sup_N \Vert S_N(x)\Vert/\Vert x\Vert<\infty$ $\forall x\in\mathbb X$, where $S_N(x) = \sum_{j=1}^N e_j^*(x)e_j$ denotes the algorithm of the partial sums.
	
	As usual  $\supp(x)=\{n\in \N: e_n^*(x)\ne 0\}$, given a finite set $A\subset \mathbb N$, $|A|$ denotes the cardinality of the set $A$, $P_A$ is the projection operator, that is, $P_A(\sum_j a_j e_j)=\sum_{j\in A} a_j e_j$, $P_{A^c}=\text{I}_{\mathbb X}-P_A$, $\mathbf{1}_{\varepsilon A}=\sum_{n\in A} \varepsilon_ne_n$ with $\vert \varepsilon_n\vert = 1$ (where $\varepsilon_n$ could be real or complex), $\mathbf{1}_A=\sum_{n\in A}e_n$ and for $A,B\subset\mathbb N$, we write $A<B$ if $\max_{i\in A} i < \min_{j\in B} j$.
	
	In 1999, S. V. Konyagin and V. N. Temlyakov introduced the \textit{Thresholding Greedy Algorithm} (TGA) (see \cite{KT}): given $x = \sum_{i=1}^{\infty}e_i^{*}(x)e_i \in \X$, we define the \textit{natural greedy ordering} for $x$ as the map $\rho: \mathbb{N}\longrightarrow\mathbb{N}$ such that $\supp(x) \subset \rho(\mathbb{N})$ and so that if $j<k$ then either $\vert e_{\rho(j)}^*(x)\vert > \vert e_{\rho(k)}^*(x)\vert$ or $\vert e_{\rho(j)}^*(x)\vert = \vert e_{\rho(k)}^*(x)\vert$ and $\rho(j)<\rho(k)$. The $m$-\textit{th greedy sum} of $x$ is $$\mathcal{G}_m(x) = \sum_{j=1}^m e_{\rho(j)}^*(x)e_{\rho(j)},$$
	and the sequence of maps $(\mathcal{G}_m)_{m=1}^\infty$ is known as the \textit{Thresholding Greedy Algorithm} associated to $\mathcal{B}$ in $\X$. 	Alternatively we can write $\mathcal G_m(x) = \sum_{k \in A_m(x)} e_k^*(x) e_k$, where $A_m(x) = \{ \rho(n) : n \leq m\}$
	is the \textit{greedy set} of $x$: $\min_{k \in A_m(x)} \vert e_k^*(x) \vert \geq \max_{k \notin A_m(x)} \vert e_k^*(x) \vert$.
	
	To study the efficiency of the TGA,  S. V. Konyagin and V. N. Temlyakov introduced in \cite{KT} the so called \textit{greedy bases}. 
	\begin{defn}
		We say that $\mathcal B$ is \textit{greedy} if there exists a constant $C\geq 1$ such that
		\begin{eqnarray*}\label{greedy}
	\Vert x-\mathcal G_m(x)\Vert \leq C\sigma_m(x),\; \forall x\in\mathbb X, \forall m\in\mathbb N,
		\end{eqnarray*}
		where $\sigma_m(x)$ is the $m$-th error of approximation with respect to $\mathcal B$, and it is defined as $$\sigma_m(x,\mathcal B)_{\mathbb X}=\sigma_m(x):=\inf\left\lbrace \left\Vert x-\sum_{n\in C}a_n e_n\right\Vert : \vert C\vert = m, a_n\in\mathbb F\right\rbrace.$$
	\end{defn}

Also, S. V. Konyagin and V. N. Temlyakov characterized greedy bases in terms of \textit{unconditional} bases with the additional property of being \textit{democratic}, i.e,  $\Vert \mathbf{1}_{A}\Vert \leq C_{d}\Vert \mathbf{1}_B\Vert$ for any pair of finite sets $A,B$ with $\vert A\vert \leq \vert B\vert$. Recall that a basis $\mathcal{B}$ in $\X$ is called unconditional if  any rearrangement of the series $\sum_{n=1}^{\infty}e_{n}^{*}(x)e_n$ converges in norm to $x$ for any $x\in \X$. This turns out to be  equivalent the fact that the projections $P_A$
are uniformly bounded on all finite sets $A$, 
i.e. there exists a constant $C\geq 1$ such that
\begin{eqnarray*}\label{sup}
\Vert P_A (x)\Vert \leq C\Vert x\Vert,\;\; \forall x\in\X \hbox{ and }\; \forall A\subset\mathbb{N}.
\end{eqnarray*}

Another important concept in greedy approximation theory is the notion of \textit{quasi-greedy} bases introduced in \cite{KT}. 
\begin{defn}
We say that $\mathcal{B}$ is \textit{quasi-greedy} if there exists a constant $C\geq 1$ such that
\begin{eqnarray}\label{quasi}
\Vert x-\mathcal G_m(x)\Vert \leq C\Vert x\Vert,\; \forall x\in\mathbb X, \forall m\in\mathbb N.
\end{eqnarray}
We denote by $C_q$ the least constant that satisfies \eqref{quasi} and we say that $\mathcal B$ is $C_q$-quasi-greedy.
\end{defn}
Subsequently, P. Wojtaszczyk proved in \cite{Woj} that $\mathcal B$ is quasi-greedy in a quasi-Banach space $\mathbb X$ if and only if the algorithm converges, that is,
$$\lim_{m\rightarrow\infty}\Vert x-\mathcal G_m(x)\Vert = 0,\; \forall x\in\mathbb X.$$

One intermediate concept between greedy and quasi-greedy bases, \textit{almost-greedy} bases, was introduced by S. J. Dilworth et al. in \cite{DKKT}. 
\begin{defn}
	We say that $\mathcal B$ is \textit{almost-greedy} if there exists a constant $C\geq 1$ such that
	\begin{eqnarray}\label{almost}	
	\Vert x-\mathcal{G}_m(x)\Vert \leq C\tilde{\sigma}_m(x),\; \forall x\in\mathbb X,\forall m \in \mathbb N,
	\end{eqnarray}
	where $\tilde{\sigma}_m(x,\mathcal B)_{\mathbb X}=\tilde{\sigma}_m(x):=\inf\lbrace \Vert x-P_A(x)\Vert : \vert A\vert = m\rbrace$.
	We denote by $C_{al}$ the least constant that satisfies \eqref{almost} and we say that $\mathcal B$ is $C_{al}$-almost-greedy.
\end{defn}
In \cite{DKKT}, the authors characterized the almost-greedy bases in terms of quasi-greedy and democratic bases. 

\begin{thm}{\cite[Theorem 3.3]{DKKT}}\label{chal}
	$\mathcal B$ is almost-greedy if and only if $\mathcal B$ is quasi-greedy and democratic.
\end{thm}

We will use the notion of \textit{super-democracy} instead of democracy. This is a classical concept in this theory.
\begin{defn}
We say that $\mathcal B$ is \textit{super-democratic} if there exists a constant $C\geq 1$ such that
\begin{eqnarray}\label{super}
\Vert \mathbf{1}_{\varepsilon A}\Vert \leq C\Vert \mathbf{1}_{\eta B}\Vert,
\end{eqnarray}
for any pair of finite sets $A$ and $B$ such that $\vert A\vert\leq \vert B\vert$ and any choice $\vert \varepsilon\vert = \vert \eta \vert = 1$.
We denote by $C_{sd}$ the least constant that satisfies \eqref{super} and we say that $\mathcal B$ is $C_{sd}$-super-democratic.
\end{defn}

\begin{rem}\label{remark}
It is well known that in Theorem \ref{chal} we can replace democracy by super-democracy (see for instance \cite[Theorem 1.3]{BBG}).
\end{rem}

On the other hand, S. J. Dilworth, N. J. Kalton and D. Kutzarova introduced in \cite{DKK} the concept of \textit{semi-greedy} bases. This concept was born as an enhancement of the TGA to improve the rate of convergence. To study the notion of semi-greediness, we need to define the \textit{Thresholding Chebyshev Greedy Algorithm}: let $A_m(x)$ be the greedy set of $x$ of cardinality $m$. Define the \textit{$m$-th Chebyshev-greedy sum} as any element $\mathcal{CG}_m(x)\in span\lbrace e_i : i\in A_m(x)\rbrace$ such that 
$$\Vert x-\mathcal{CG}_m(x)\Vert = \min\left\lbrace \left\Vert x-\sum_{n\in A_m(x)}a_n e_n\right\Vert : a_n\in\mathbb F\right\rbrace.$$
The collection $\lbrace \mathcal{CG}_m\rbrace_{m=1}^\infty$ is the \textit{Thresholding Chebyshev Greedy Algorithm}.

\begin{defn}
	We say that $\mathcal B$ is \textit{semi-greedy} if there exists a constant $C\geq 1$ such that
	\begin{eqnarray}\label{semig}
	\Vert x-\mathcal{CG}_m(x)\Vert \leq C\sigma_m(x),\; \forall x\in\mathbb X, \forall m\in\mathbb N.
	\end{eqnarray}
	We denote by $C_s$ the least constant that satisfies \eqref{semig} and we say that $\mathcal B$ is $C_s$-semi-greedy.
\end{defn}

In \cite{DKK}, the following theorem is proved:
\begin{thm}{\cite[Theorem 3.2]{DKK}}\label{dkk}
Every almost-greedy basis in a Banach space is semi-greedy.
\end{thm}

In this paper we study the converse of this theorem. In \cite{DKK}, the authors established the following "converse" theorem:
 \begin{thm}{\cite[Theorem 3.6]{DKK}}
 	Assume that $\mathcal B$ is a semi-greedy basis in a Banach space $\mathbb X$ which has finite cotype. Then, $\mathcal B$ is almost-greedy.
 \end{thm}

The objective here is to show that the condition of the finite cotype in the last theorem is not necessary. The main result is the following:
\begin{thm}\label{equiv}
	Assume that $\mathcal B$ is a Schauder basis in a Banach space $\X$. 
	\begin{itemize}
		\item[a)] If $\mathcal B$ is $C_{q}$-quasi-greedy and $C_{sd}$-super-democratic, then $\mathcal B$ is $C_s$-semi-greedy with constant $C_s\leq C_q+4C_qC_{sd}$.
		\item[b)] If $\mathcal B$ is $C_s$-semi-greedy, then $\mathcal B$ is $C_{sd}$-super-democratic with constant $C_{sd}\leq 2(C_sK_b)^2$ and $C_{q}$-quasi-greedy with constant $C_{q}\leq K_b(2+3(K_bC_s)^2)$.
	\end{itemize}

\end{thm}
 \begin{rem}\label{remark2}
	S. J. Dilworth et al. (\cite{DKK}) proved the item $a)$ with the bound $C_s=O(C_q^2C_d)$, where $C_d$ is the democracy constant. Here, we slightly relax this bound proving that $C_s=O(C_{q}C_{sd})$. 
\end{rem}
 \begin{cor}\label{charalmo}
 	If $\mathcal B$ is a Schauder basis in $\mathbb X$, $\mathcal B$ is almost-greedy if and only if $\mathcal B$ is semi-greedy.
 \end{cor}

		\end{section}
	
	\section{Preliminary results}\label{lemmas}
	To prove Theorem \ref{equiv}, we need the following technical results that we can find in \cite{BBG} and \cite{DKKT}.
	\subsection{Convexity lemma}
	\begin{lem}{\cite[Lemma 2.7]{BBG}}\label{conv}
	For every finite set $A\subset \mathbb N$, we have
	$$\text{co}\lbrace \mathbf{1}_{\varepsilon A} : \vert \varepsilon\vert=1\rbrace = \left\lbrace \sum_{n\in A}z_n e_n : \vert z_n\vert \leq 1\right\rbrace,$$
	where $\text{co}S=\lbrace \sum_{j=1}^n \alpha_j x_j : x_j \in S, 0\leq \alpha_j\leq 1, \sum_{j=1}^n \alpha_j = 1, n\in\mathbb N\rbrace$.
	\end{lem}

As a consequence, for any finite sequence $(z_n)_{n\in A}$ with $z_n\in\mathbb F$ for all $n\in A$,
		$$\left\Vert \sum_{n\in A}z_n e_n\right\Vert \leq \max_{n\in A}\vert z_n\vert \varphi(\vert A\vert),$$
		where $\varphi(m)=\sup_{\vert A\vert= m, \vert\varepsilon\vert=1}\Vert \mathbf{1}_{\varepsilon A}\Vert$.

\subsection{The truncation operator}
For each $\alpha>0$, we define the \textit{truncation function} of $z\in\mathbb F$ as
$$T_\alpha (z) = \alpha\sgn(z),\; \vert z\vert > \alpha,\;\; T_\alpha(z) = z,\; \vert z\vert \leq\alpha.$$
We can extend $T_\alpha$ to an operator in $\X$ by
$$T_\alpha(x)=\sum_{i=1}^\infty T_\alpha(e_i^*(x))e_i = \alpha\mathbf{1}_{\varepsilon\Gamma_\alpha}+P_{\Gamma_\alpha^c}(x),$$
where $\Gamma_\alpha = \lbrace n : \vert e_n^*(x)\vert > \alpha\rbrace$ and $\varepsilon_j = \sgn(e_j^*(x))$ with $j\in \Gamma_\alpha$.
Hence, this is a well-defined operator for all $x\in \X$ since $\Gamma_\alpha$ is a finite set.

This operator was introduced in \cite{DKK} to prove Theorem \ref{dkk} showing that for quasi-greedy bases, this operator is uniformly bounded. A slight improvement of the boundedness constant was given in \cite{BBG}.

\begin{prop}{\cite[Lemma 2.5]{BBG}}\label{truncation}
	Assume that $\mathcal B$ is $C_q$-quasi-greedy basis in a Banach space $\X$. Then, for every $\alpha>0$,
	$$\Vert T_\alpha(x)\Vert\leq C_q\Vert x\Vert,\; \forall x\in\X.$$
\end{prop}

We shall also use the following known inequality from  \cite{DKKT}.
\begin{lem}{\cite[Lemma 2.2]{DKKT}}\label{propC}
If $\mathcal B$ is a $C_q$-quasi-greedy basis in $\mathbb X$, 
\begin{eqnarray}\label{ineqC}
\min_{j\in G}\vert e_j^*(x)\vert \Vert \mathbf{1}_{\varepsilon G}\Vert \leq 2C_q\Vert x\Vert,\; \forall x\in\mathbb X, \forall G\; \text{greedy set of}\; x,
\end{eqnarray}
with $\varepsilon = \lbrace\sgn(e_j^*(x))\rbrace$.
\end{lem}

\section{Proof of the main result}
Using the lemmas of Section \ref{lemmas}, we prove Theorem \ref{equiv}.
\begin{proof}[Proof of Theorem \ref{equiv}]
First, we show the proof of $a)$. Suppose that $\mathcal B$ is $C_{q}$-quasi-greedy and $C_{sd}$-super-democratic. To show the semi-greediness, we will follow the same procedure as in the proof of \cite[Theorem 4.1]{DKO} and \cite[Theorem 3.2]{DKK}.
Take $x\in \X$ and $z=\sum_{i\in B}a_i e_i$ with $\vert B\vert=m$ such that $\Vert x-z\Vert <\sigma_m(x)+\delta$, for $\delta>0$. Let $A_m(x)$ the greedy set of $x$ of cardinality $m$. We write $x-z:= \sum_{i=1}^\infty y_i e_i$, where $y_i = e_i^*(x)-a_i$ for $i\in B$ and $y_i=e_i^*(x)$ for $i\not\in B$.
To prove that $\mathcal B$ is semi-greedy we only have to show that there exists $w\in\X$ so that $supp(x-w)\subset A_m(x)$ and $\Vert w\Vert \leq c \Vert x-z\Vert$ for some positive constant $c$. If $\alpha = \max_{j\not \in A_m(x)}\vert e_j^*(x)\vert$, we take the element $w$ as is defined in \cite{DKK}:
$$w:=\sum_{i\in A_m(x)}T_\alpha(y_i)e_i + P_{A_m^c(x)}(x) = \sum_{i=1}^\infty T_\alpha(y_i)e_i + \sum_{i\in B\setminus A_m(x)}(e_i^*(x)-T_\alpha(y_i))e_i.$$
Of course, $w$ satisfies that $\supp(x-w)\subset A_m(x)$ and we will prove that $\Vert w\Vert \leq (C_q + 4C_qC_s)\Vert x-z\Vert$. To obtain this bound, using Proposition \ref{truncation},
\begin{eqnarray}\label{f1}
\Vert \sum_{i=1}^\infty T_\alpha(y_i)e_i\Vert \leq C_q\Vert x-z\Vert.
\end{eqnarray}

Taking into account that $\vert e_i^*(x)-T_\alpha(y_i)\vert \leq 2\alpha$ for $i\in B\setminus A_m(x)$, using Lemma \ref{conv},
\begin{eqnarray}\label{f2}
\left\Vert \sum_{i\in B\setminus A_m(x)}(e_i^*(x)-T_\alpha(y_i))e_i\right\Vert \leq 2\alpha\varphi(\vert B\setminus A_m(x)\vert)\leq 2\min_{j\in A_m(x)\setminus B}\vert e_j^*(x-z)\vert \varphi(\vert A_m(x)\setminus B\vert).
\end{eqnarray}
To improve the bound of $C_s$ as we have commented in the Remark \ref{remark2}, based on (\cite[Lemma 2.1]{GHO}), we can find a greedy set $\Gamma$ of $x-z$ with the following conditions: 
\begin{itemize}
	\item $\vert \Gamma\vert = \vert B\setminus A_m(x)\vert$,
	\item $\min_{j\in A_m(x)\setminus B}\vert e_j^*(x-z)\vert \leq \min_{j\in \Gamma}\vert e_j^*(x-z)\vert$.
\end{itemize}
Hence, using $\varepsilon = \lbrace\sgn(e_j^*(x-z))\rbrace$ and Lemma \ref{propC},
\begin{eqnarray}\label{f3}
\min_{j\in A_m(x)\setminus B}\vert e_j^*(x-z)\vert\varphi(\vert B\setminus A_m(x)\vert) \leq C_{sd}\min_{j\in \Gamma}\vert e_j^*(x-z)\vert\Vert \mathbf{1}_{\varepsilon\Gamma}\Vert\leq 2C_qC_{sd}\Vert x-z\Vert.
\end{eqnarray}

Thus, using \eqref{f1}, \eqref{f2}, \eqref{f3}, the basis is $C_s$-semi-greedy with constant $C_s\leq (C_q+4C_qC_{sd})$.
\newline

Now, we prove $b)$. Assume that $\mathcal B$ is $C_s$-semi-greedy.

Super-democracy can be proved using the technique of \cite[Proposition 3.3]{DKK}. Indeed, take $A$ and $B$ with $\vert A\vert \leq \vert B\vert$ and $\vert \varepsilon\vert=\vert \eta \vert=1$. Select now a set $D$ such that $\vert D\vert = \vert A\vert$, $D>(A\cup B)$ and define $z:= \mathbf{1}_{\varepsilon A}+(1+\delta)\mathbf{1}_D$ with $\delta>0$ . It is clear that $\mathcal G_{\vert D\vert}(z)=(1+\delta)\mathbf{1}_D$. Then,
	$$\Vert z-\mathcal{CG}_{\vert D\vert}(z)\Vert = \left\Vert \mathbf{1}_{\varepsilon A}+\sum_{i\in D}c_i e_i\right\Vert,$$
	where the scalars $(c_i)_{i\in D}$ are given by the Chebyshev approximation. Then,
	$$\Vert \mathbf{1}_{\varepsilon A}\Vert \leq K_b\Vert \mathbf{1}_{\varepsilon A}+\sum_{i\in D}c_i e_i\Vert \leq K_b C_s\sigma_{\vert D\vert}(z)\leq K_bC_s\Vert (1+\delta)\mathbf{1}_D\Vert.$$
	If $\delta$ goes to $0$, 
	\begin{eqnarray}\label{sup1}
	\Vert \mathbf{1}_{\varepsilon A}\Vert \leq C_sK_b\Vert \mathbf{1}_D\Vert. 
	\end{eqnarray}
	The next step is to obtain that $\Vert \mathbf{1}_D\Vert \leq 2K_bC_s\Vert\mathbf{1}_{\eta B}\Vert$. For that, we take the element $y:= (1+\delta)\mathbf{1}_{\eta B}+\mathbf{1}_D$ with $\delta>0$. Then, $\mathcal G_{\vert B\vert}(y)=(1+\delta)\mathbf{1}_{\eta B}$. Hence,
	$$\Vert y-\mathcal{CG}_{\vert B\vert}(y)\Vert=\left\Vert \sum_{i\in B}d_i e_i + \mathbf{1}_D\right\Vert,$$
	where as before, the scalars $(d_i)_{i\in B}$ are given by the Chebyshev approximation. Using again the semi-greediness,
	$$\Vert \mathbf{1}_D\Vert \leq 2K_b\Vert \sum_{i\in B}d_i e_i + \mathbf{1}_D\Vert \leq 2C_sK_b\sigma_{\vert B\vert}(y)\leq 2C_sK_b\Vert (1+\delta)\mathbf{1}_{\eta B}\Vert.$$
	Taking $\delta\rightarrow 0$, we obtain that
	
	\begin{eqnarray}\label{sup2}
	\Vert \mathbf{1}_D\Vert \leq 2C_sK_b\Vert \mathbf{1}_{\eta B}\Vert.	
	\end{eqnarray}
 
Using \eqref{sup1} and \eqref{sup2}, $$\Vert \mathbf{1}_{\varepsilon A}\Vert \leq 2(C_sK_b)^2\Vert \mathbf{1}_{\eta B}\Vert.$$
	Hence, the basis is super-democratic with constant $C_{sd}\leq 2(C_sK_b)^2$.
	
	To prove now the quasi-greediness, we will present a more elemental proof than in \cite[Theorem 3.6]{DKK} that works for general Banach spaces: take an element $x\in\X$ with finite support and $A_m(x)$ the greedy set of $x$ with cardinality $m$, take $D>\supp(x)$ with $\vert D\vert = \vert A_m(x)\vert=m$ and define $z:= x-\mathcal G_m(x)+(\delta+\alpha) \mathbf{1}_D$, where $\delta>0$ and $\alpha = \min_{j\in A_m(x)}\vert e_j^*(x)\vert$. Then, since $A_m(z)=D$, $$\Vert z-\mathcal{CG}_{m}(z)\Vert = \left\Vert x-\mathcal G_m(x)+\sum_{i\in D}f_i e_i\right\Vert,$$
	for some scalars $(f_i)_{i\in D}$ given by the Chebyshev approximation. Then,
	\begin{eqnarray*}
	\Vert x-\mathcal G_m(x)\Vert \leq K_b\left\Vert x-\mathcal G_m(x)+\sum_{i\in D}f_i e_i\right\Vert\leq K_bC_s\sigma_{m}(z)\leq K_bC_s\Vert x+(\delta+\alpha) \mathbf{1}_D\Vert.
	\end{eqnarray*}
	Taking $\delta \rightarrow 0$, 
	\begin{eqnarray}\label{qg1}
	\Vert x-\mathcal G_m(x)\Vert \leq K_bC_s\Vert x+\alpha \mathbf{1}_D\Vert\leq K_bC_s(\Vert x\Vert +\Vert \alpha \mathbf{1}_D\Vert).
	\end{eqnarray}
	
	Select now $y:= \sum_{j\in A_m(x)}(e_j^*(x)+\delta\varepsilon_j)e_j+ \sum_{j\in A_m^c(x)}e_j^*(x)e_j+\alpha \mathbf{1}_D$, with $\delta>0$ and $\varepsilon_j = \sgn(e_j^*(x))$ for $j\in A_m(x)$. Then, since $\mathcal{G}_{m}(y)= \sum_{j\in A_m(x)}(e_j^*(x)+\delta\varepsilon_j)e_j$, using Chebyshev approximation,
	$$\Vert y -\mathcal{CG}_{m}(y)\Vert = \left\Vert \sum_{j\in A_m(x)}a_i e_i + \sum_{j\in A_m^c(x)}e_j^*(x)e_j+\alpha \mathbf{1}_D\right\Vert.$$
	Hence,
	\begin{eqnarray*}
	\Vert \alpha \mathbf{1}_D\Vert &\leq& 2K_b\left\Vert \sum_{j\in A_m(x)}a_i e_i + \sum_{j\in A_m^c(x)}e_j^*(x)e_j +\alpha \mathbf{1}_D\right\Vert\leq 2K_bC_s\sigma_{m}(y)\\
	&\leq& 2K_bC_s\left\Vert \sum_{j\in A_m(x)}(e_j^*(x)+\delta\varepsilon_j)e_j+ \sum_{j\in A_m^c(x)}e_j^*(x)e_j\right\Vert.
	\end{eqnarray*}
	Taking $\delta\rightarrow 0$, $\Vert \alpha \mathbf{1}_D\Vert \leq 2K_bC_s\Vert x\Vert$. Using the last inequality and \eqref{qg1},
	$$\Vert x-\mathcal G_m(x)\Vert \leq K_bC_s(\Vert x\Vert + 2K_bC_s\Vert x\Vert)\leq 3(K_bC_s)^2\Vert x\Vert.$$
	Thus, $\Vert x-\mathcal{G}_m(x)\Vert \leq 3(K_bC_s)^2\Vert x\Vert$ for any finite $x\in\X$ and $m\leq \vert\supp(x)\vert$. 
	
For the general case, we take  $x\in\mathbb X$ and $A_m(x)$ the greedy set of $x$ with cardinality $m$. We can find a number $N\in\mathbb N$ such that $A_m(x)\subset \lbrace 1,...,N\rbrace$. Then, since  $\mathcal G_m(x)=\mathcal G_m(S_N(x))$, applying that $\mathcal B$ is Schauder and quasi-greedy for elements with finite support,
\begin{eqnarray*}
\Vert x-\mathcal G_m(x)\Vert&\leq&  \Vert x-S_N(x)\Vert + \Vert S_N(x)- \mathcal G_m(x)\Vert\\
&=& \Vert x-S_N(x)\Vert + \Vert S_N(x)-\mathcal{G}_m(S_N(x))\Vert\\
&\leq& 2K_b\Vert x\Vert + 3(K_bC_s)^2\Vert S_N(x)\Vert\\
&\leq& K_b(2+3(K_bC_s)^2)\Vert x\Vert.
\end{eqnarray*}

This completes the proof.
\end{proof}

\begin{proof}[Proof of Corollary \ref{charalmo}]: The proof follows using Theorem \ref{equiv}, Theorem \ref{chal} and Remark \ref{remark}.
\end{proof}
	
\begin{rem}
In \cite[Section 6-Question 3]{BDKOW}, the authors ask the following question: if a basis $\mathcal B$ satisfies Property (A) and the inequality \eqref{ineqC}, is $\mathcal B$ semi-greedy? We remind that $\mathcal B$ satisfies Property (A) if there is a positive constant $C_a$ such that
$$\Vert x+\mathbf{1}_{\varepsilon A}\Vert \leq C_{a} \Vert x+\mathbf{1}_{\eta B}\Vert,$$
for any $x\in\X$, $A, B$ such that $\vert A\vert = \vert B\vert<\infty$, $A\cap B = \emptyset$, $(A\cup B)\cap \supp(x)= \emptyset$, $\vert\varepsilon\vert=\vert \eta\vert=1$ and $\max_j \vert e_j^*(x)\vert \leq 1$. 
The answer is not due to the example in \cite[Subsection 5.5]{BBG} of a basis $\mathcal B$ in a Banach space such that $\mathcal B$ satisfies the Property (A) and \eqref{ineqC}, but is not quasi-greedy, hence is not almost-greedy and using Theorem \ref{equiv}, $\mathcal B$ is not semi-greedy.
\end{rem}

\section{Open questions}
As discussed in \cite{Woj} (see also \cite{DKO}), one can define the Thresholding Greedy Algorithm and the Thresholding Chebyshev Greedy Algorithm in the context of Markushevich bases, that is, $\lbrace e_i, e_i^*\rbrace$ is a semi-normalized biorthogonal system, $\mathbb X=\overline{span\lbrace e_i : i\in \mathbb N\rbrace}^{\mathbb X}$ and $\mathbb X^* = \overline{span\lbrace e_i^* : i\in\mathbb N\rbrace}^{w^*}$. 
In section a) of Theorem \ref{equiv}, it is enough to work with Markushevich bases instead of Schauder bases. However, in the item b), seems to be necessarily to use that $\mathcal B$ is Schauder to prove the result. 

\textbf{Question 1}: Is it possible to remove the condition to be Schauder in section b) of Theorem \ref{equiv}?

Another interesting problem is to establish if almost-greediness implies the condition to be Schauder. Of course, if $\mathcal B$ is greedy then $\mathcal B$ is Schauder since greediness implies unconditionality. As far as we know, all of examples of almost-greedy bases in the literature seem to be Schauder bases, but we don't know if almost-greediness implies that $\mathcal B$ is Schauder or not.

\textbf{Question 2}: If $\mathcal B$ is almost-greedy, is it necessarily Schauder?

%
%
\medskip

\textbf{Acknowledgments}: Thanks to Eugenio Hern\'andez, Gustavo Garrig\'os, Fernando Albiac and Jos\'e Luis Ansorena for many interesting discussions during the elaboration of this paper.

	\medskip

\end{document}